\documentclass[reqno,12pt,a4paper]{amsart}
\usepackage{amscd,amsfonts,amssymb,amsthm,latexsym}
\usepackage{mathtools}
\usepackage{bm}
\usepackage{srcltx}

\theoremstyle{plain}
\newtheorem{theorem}{Theorem}[section]
\newtheorem*{theorem*}{Theorem}

\newtheorem{lemma}{Lemma}[section]

\theoremstyle{definition}

\newtheorem*{definition*}{Definition}

\theoremstyle{remark}

\newtheorem*{remark*}{Remark}

\numberwithin{equation}{section}

\begin{document}
\raggedbottom 

\title[Sums related to Euler's totient function]{Sums related to Euler's totient function}

\author{Artyom Radomskii}

\begin{abstract}We obtain an upper bound for the sum $\sum_{n\leq N} (a_{n}/\varphi (a_{n}))^{s}$, where $\varphi$ is Euler's totient function, $s\in \mathbb{N}$, and $a_{1},\ldots, a_{N}$ are positive integers (not necessarily distinct) with some restrictions. As applications, for any $t>0$, we obtain an upper bound for the number of $n\in [1,N]$ such that $a_{n}/ \varphi (a_{n})> t$.
\end{abstract}

 \address{HSE University, Moscow, Russian Federation}

 \email{artyom.radomskii@mail.ru}

\keywords{Euler's totient function, multiplicative functions, prime numbers.}

\maketitle

\section{Introduction}

Let $\varphi$ denote Euler's totient function. We prove

\begin{theorem}\label{T_MULTI}
Let $a_{1},\ldots, a_{N}$ be positive integers \textup{(}not necessarily distinct\textup{)}, $a_{n}\leq M$ for all $1\leq n \leq N$. For $d\in \mathbb{N}$ we set
\[
\omega(d)=\#\{n\leq N: a_{n}\equiv 0\text{ \textup{(mod $d$)}}\}.
\]Let $\alpha\in (0, 1]$, $y=\max ((\log M)^{\alpha}, 2)$, and $\mathcal{D}$ be the collection of square-free numbers, all of whose prime divisors lie in $(1,y]$ \textup{(}we note that $1\in \mathcal{D}$\textup{)}. Let
\[
\omega (n) \leq  K g(n)
\]for any $n\in \mathcal{D}$, where $K>0$ is a constant (depending on $N$) and $g(n)$ is a multiplicative function. Then
\[
\sum_{n=1}^{N} \bigg(\frac{a_{n}}{\varphi (a_{n})}\bigg)^{s}\leq K \Big(\frac{c}{\alpha}\Big)^{s} \prod_{p\leq y}
\Big(1 + ((1+p^{-1})^{s}-1)g(p)\Big)
\]for any $s\in \mathbb{N}$. Here $c > 0$ is an absolute constant.
\end{theorem}

\begin{theorem}\label{T1}

Suppose that the assumptions of Theorem \ref{T_MULTI} hold, $K = \gamma N$, $0<g(p)\leq c_{0}$ for any prime $p$, and
\[
L=\sum_{p}\frac{g(p)}{p}<\infty.
\] Then there are positive constants $C$, $c_{1}$, and $c_{2}$ depending only on $\gamma$, $L$, $c_{0}$, and $\alpha$ such that
\begin{equation}\label{T1:EQ.1}
\sum_{n=1}^{N} \bigg(\frac{a_{n}}{\varphi (a_{n})}\bigg)^{s}\leq \textup{exp} (s \log\log (s+2) + Cs)N
\end{equation}for any $s\in \mathbb{N}$, and
\begin{equation}\label{T1:EQ.2}
\#\Big\{n\leq N: \frac{a_{n}}{\varphi (a_{n})}>t\Big\}\leq c_{1} \textup{exp}(-\textup{exp}(c_{2}t)) N
\end{equation}for any $t>0$.
\end{theorem}

From Theorems \ref{T_MULTI} and \ref{T1} we obtain the following result.

\begin{theorem}\label{C2}
Let $f(n)=b_{d}n^d +\ldots + b_{0}$ be a polynomial with integer coefficients such that $b_{d}>0$, $(b_{d},\ldots, b_{0})=1$, and $f: \mathbb{N}\to \mathbb{N}$. Then there exist positive constants $C$, $c_{1}$, and $c_{2}$ depending only on $f$ such that for any $x\geq 1$, $s\in \mathbb{N}$, and $t>0$ we have
\begin{equation}\label{C2:EQ.1}
\sum_{n\leq x}\Big(\frac{f(n)}{\varphi(f(n))}\Big)^{s}\leq \textup{exp}(s\log\log (s+2) + Cs) x
\end{equation}
and
\begin{equation}\label{C2:EQ.2}
\#\Big\{n\leq x: \frac{f(n)}{\varphi (f(n))}>t\Big\} \leq c_{1} \textup{exp}(-\textup{exp}(c_{2}t)) x.
\end{equation}
\end{theorem}

 Theorem \ref{C2} extends Corollary 1.1 in \cite{Radomskii.Izv} which showed the inequality \eqref{C2:EQ.1} but with an upper bound $\textup{exp} (s \log s + Cs) x$ and without the inequality \eqref{C2:EQ.2}.

Let $\mathcal{L}=\{L_1,\ldots,L_k\}$ be a set of $k$ distinct linear functions with integer coefficients
\[
L_{i}(n)=a_i n+b_i,\qquad i=1,\ldots, k.
\]For $L(n)=an+b$, $a, b\in \mathbb{Z}$, we define
\[
\Delta_{L}=|a|\prod_{i=1}^{k}|a b_i - b a_i|.
\]We note that if $a_i = a$ for all $1 \leq i \leq k$, then
 \[
\Delta_{L}=|a|^{k+1}\prod_{i=1}^{k}|b - b_i|.
\] Modern application of the sieve methods involves the sums
\[
\sum_{(a,b)\in \Omega} \frac{\Delta_{L}}{\varphi(\Delta_{L})}
\](see, for example, \cite{Maynard}). Here $(a,b)$ denotes a vector and $\Omega$ is a finite set in $\mathbb{Z}^2$. We prove

\begin{theorem}\label{C_LIN}
Let $b_1, \ldots, b_k$ be distinct integers, $a\in \mathbb{N}$, $s\in \mathbb{N}$. Let $k\geq 2,$ $x \geq 3$, $(\log x)^{-9/10} \leq \eta \leq 1$, and $|b_i|\leq \log x$ for all $1 \leq i \leq k$. We set
\[
f(b) = (b-b_1) \dots (b-b_k).
\] Then the following statements hold.

\textup{1)} If $k\geq \log\log x$ or $s \leq k$, then
\begin{equation}\label{C_LIN:EQ.1}
\sum_{\substack{|b|\leq \eta \log x\\ b\,\neq\,b_1,\ldots, b_k}} \bigg(\frac{a^{k+1}|f(b)|}{\varphi(a^{k+1}|f(b)|)}\bigg)^{s} \leq \Big(c\,\frac{a}{\varphi (a)} \log k\Big)^{s} \eta \log x.
\end{equation}

\textup{2)} If $2 \leq k< \log\log x$ and $s > k$, then
\begin{equation}\label{C_LIN:EQ.2}
\sum_{\substack{|b|\leq \eta \log x\\ b\,\neq\,b_1,\ldots, b_k}} \bigg(\frac{a^{k+1}|f(b)|}{\varphi(a^{k+1}|f(b)|)}\bigg)^{s} \leq \Big(c\,\frac{a}{\varphi (a)} \log s\Big)^{s} \eta \log x.
\end{equation} Here $c>0$ is an absolute constant.
\end{theorem}

 Theorem \ref{C_LIN} extends a result of Maynard (\cite[Lemma 8.1]{Maynard}) which showed the same result but with $s=1$ and extends Theorem 1.4 in \cite{Radomskii.Izv} which showed the inequality \eqref{C_LIN:EQ.2} but with an upper bound
\[
\Big(c\,\frac{a}{\varphi (a)} \log k\Big)^{s} s!\, \eta \log x.
\]

From Theorem \ref{T1} we obtain
\begin{theorem}\label{T.f(p)}

Let $f(n)=b_{d}n^d +\ldots + b_{0}$ be a polynomial with integer coefficients such that $b_{d}>0$, $(b_{d},\ldots, b_{0})=1$, and $f: \mathbb{P}\to \mathbb{N}$. Then there exist positive constants $C$, $c_{1}$, and $c_{2}$ depending only on $f$ such that for any $x\geq 2$, $s\in \mathbb{N}$, and $t>0$ we have
\begin{equation}\label{T.f(p):EQ.1}
\sum_{p\leq x}\Big(\frac{f(p)}{\varphi(f(p))}\Big)^{s}\leq \textup{exp}(s\log\log (s+2) + Cs) \pi(x)
\end{equation}
and
\begin{equation}\label{T.f(p):EQ.2}
\#\Big\{p\leq x: \frac{f(p)}{\varphi (f(p))}>t\Big\} \leq c_{1} \textup{exp}(-\textup{exp}(c_{2}t)) \pi(x).
\end{equation}
\end{theorem} In particular, we can take $f(x)= x-1$ or $f(x)=x^2 + 1$ in Theorem \ref{T.f(p)}.

\section{Notation}

We reserve the letter $p$ for primes. In particular, the sum $\sum_{p\leq K}$ should be interpreted as being over all prime numbers not exceeding $K$. We denote the number of primes not exceeding $x$ by $\pi (x)$ and the number of primes not exceeding $x$ that are congruent to $a$ modulo $k$ by $\pi (x; k, a)$. Let $\# A$ denote the number of elements of a finite set $A$. We write $\mathbb{Z}$ for the set of all integers, $\mathbb{N}$ for the set of all positive integers, and $\mathbb{P}$ for the set of all primes. Let $(a_1,\ldots, a_n)$ be the greatest common divisor of integers $a_1,\ldots, a_n$ and $[a_{1},\ldots, a_{n}]$ their least common multiple. For real numbers $x,$ $y$ we also use $(x,y)$ to denote the open interval and $[x,y]$ to denote the closed interval. The usage of the notation will be clear from the context.

 Let $\varphi$ denote Euler's totient function, i.\,e.
\[
\varphi(n)=\#\{1\leq m \leq n: (m,n)=1\},\quad n\in \mathbb{N}.
\]
We write $\nu(n)$ for the number of distinct primes dividing $n$. By $\mu(n)$ we denote the M\"{o}bius function, which is defined to be $\mu(n)=(-1)^{\nu (n)}$ if $n$ is square-free, $\mu(n)=0$ otherwise. Let $P^{+}(n)$ denote the greatest prime factor of $n$ (by convention $P^{+}(1)=1$).

By definition, we put
\[
\sum_{\varnothing} = 0,\qquad \prod_{\varnothing}=1.
\]The symbol $b|a$ means that $b$ divides $a$. For fixed $a$ the sum $\sum_{b|a}$ and the product $\prod_{b|a}$ should be interpreted as being over all positive divisors of $a$. By $[x]$ we denote the largest integer not exceeding $x$ and $\{x\}= x-[x]$.

For a polynomial $f(x)$ with integer coefficients, by $\rho (f, m)$ we denote the number of solutions of the congruence $f(x) \equiv 0$ (mod $m$).

\section{Proof of Theorems \ref{T_MULTI} and \ref{T1}}

\begin{proof}[Proof of Theorem \ref{T_MULTI}.] We need the following result.

\begin{lemma}\label{L1}
 There is an absolute  positive constant $c$ such that if $\alpha\in (0,1]$ and $n$ is a positive integer, then
\[
\frac{n}{\varphi(n)}\leq \frac{c}{\alpha}\prod_{p|n:\,\,
p\leq (\log n)^{\alpha}} \biggl(1+\frac{1}{p}\biggr).
\]
\end{lemma}
\begin{proof} This is \cite[Lemma 3.3]{Radomskii.Izv}.
\end{proof}

For $1\leq n \leq N$, by Lemma \ref{L1} we have
\begin{align*}
 \frac{a_n}{\varphi (a_n)}&\leq \frac{c}{\alpha} \prod_{p|a_n:\,\, p\leq (\log a_n)^{\alpha}} \biggl(1+\frac{1}{p}\biggr)\\
 &\leq \frac{c}{\alpha} \prod_{p|a_n:\,\, p\leq y} \biggl(1+\frac{1}{p}\biggr)= \frac{c}{\alpha}\sum_{\substack{d|a_n:\\P^{+}(d)\leq y}}\frac{\mu^{2}(d)}{d}.
  \end{align*} Hence
  \begin{align*}
  \sum_{n=1}^{N} \Big(\frac{a_n}{\varphi (a_n)}\Big)^{s}&\leq \Big(\frac{c}{\alpha}\Big)^{s}\sum_{n=1}^{N}
  \sum_{\substack{d_{1},\ldots, d_{s}|a_n:\\P^{+}(d_{i})\leq y}}\frac{\mu^{2}(d_1)\ldots \mu^{2}(d_s)}{d_1 \ldots d_{s}}\\
  &= \Big(\frac{c}{\alpha}\Big)^{s}
  \sum_{d_{1},\ldots, d_{s}\in \mathcal{D}}\frac{\mu^{2}(d_1)\ldots \mu^{2}(d_s)}{d_1 \ldots d_{s}}\sum_{\substack{1\leq n \leq N:\\ d_1 | a_n, \ldots,\,d_{s}|a_{n}}}1\\
  &= \Big(\frac{c}{\alpha}\Big)^{s}
  \sum_{d_{1},\ldots, d_{s}\in \mathcal{D}}\frac{\mu^{2}(d_1)\ldots \mu^{2}(d_s)}{d_1 \ldots d_{s}}\,\omega ([d_{1},\ldots, d_{s}]).
  \end{align*}We obtain
  \[
  \sum_{n=1}^{N} \Big(\frac{a_n}{\varphi (a_n)}\Big)^{s} \leq \Big(\frac{c}{\alpha}\Big)^{s}
  \sum_{n\in \mathcal{D}} \omega(n)f(n),
  \]where
  \[
  f(n)=\sum_{\substack{d_{1},\ldots, d_{s}\in \mathcal{D}:\\ [d_{1},\ldots, d_{s}]=n}}\frac{\mu^{2}(d_1)\ldots \mu^{2}(d_s)}{d_1 \ldots d_{s}}.
  \]It is easy to see that $f(n)$ is supported on $\mathcal{D}$. Suppose that $m, n \in \mathcal{D}$ and $(m,n) = 1$. For any $d_1,\ldots, d_s \in \mathcal{D}$ such that $[d_1,\ldots, d_s] = mn$ there are unique $d_i^{'}, d_{i}^{''}\in \mathcal{D},$ $i=1,\ldots, s,$ such that $d_i^{'} d_{i}^{''} = d_i$, $(d_i^{'}, d_{i}^{''})=1$, $[d_1^{'},\ldots, d_s^{'}]=n$, and $[d_1^{''},\ldots, d_s^{''}]=m$. Hence
  \[
  f(nm) = \sum_{\substack{d_{1}^{'},\ldots, d_{s}^{'}\in \mathcal{D}:\\ [d_{1}^{'},\ldots, d_{s}^{'}]=n}}\frac{\mu^{2}(d_{1}^{'})\ldots \mu^{2}(d_{s}^{'})}{d_{1}^{'} \ldots d_{s}^{'}}
  \sum_{\substack{d_{1}^{''},\ldots, d_{s}^{''}\in \mathcal{D}:\\ [d_{1}^{''},\ldots, d_{s}^{''}]=m}}\frac{\mu^{2}(d_{1}^{''})\ldots \mu^{2}(d_{s}^{''})}{d_{1}^{''} \ldots d_{s}^{''}} = f(n)f(m).
  \]
  Since $f(n)=0$ for any $n\notin \mathcal{D}$, we obtain that $f(n)$ is a multiplicative function. It is clear that
   \[
  f(p)=\sum_{k=1}^{s}\binom{s}{k}p^{-k}=(1+ p^{-1})^{s} - 1
  \] for any prime $p\leq y$.

  Since $\omega(n)\leq K g(n)$ for any $n\in \mathcal{D}$, we have
  \begin{align*}
  \sum_{n=1}^{N} \Big(\frac{a_n}{\varphi (a_n)}\Big)^{s} &\leq K \Big(\frac{c}{\alpha}\Big)^{s}
  \sum_{n\in \mathcal{D}} f(n)g(n)= K\Big(\frac{c}{\alpha}\Big)^{s}
  \prod_{p\leq y} \Big(1+f(p)g(p)\Big)\notag\\
  &= K\Big(\frac{c}{\alpha}\Big)^{s}
  \prod_{p\leq y} \Big(1+((1+p^{-1})^{s}-1)g(p)\Big).
  \end{align*} Theorem \ref{T_MULTI} is proved.
  \end{proof}

\begin{proof}[Proof of Theorem \ref{T1}.] Since $g(p)>0$ for any prime $p$, from Theorem \ref{T_MULTI} we obtain
\begin{align}
\sum_{n=1}^{N} \bigg(\frac{a_{n}}{\varphi (a_{n})}\bigg)^{s}&\leq \gamma N\Big(\frac{c}{\alpha}\Big)^{s} \prod_{p\leq y}
\Big(1 + ((1+p^{-1})^{s}-1)g(p)\Big)\notag\\
&\leq \gamma N\Big(\frac{c}{\alpha}\Big)^{s} \prod_{p}
\Big(1 + ((1+p^{-1})^{s}-1)g(p)\Big).\label{T1:basic}
\end{align}

  For any prime $p>s$, by the mean value theorem there is $\xi\in (0, p^{-1})$ such that
  \begin{equation}\label{T1.LAGRANGE}
  \Big(1+\frac{1}{p}\Big)^{s}-1 = s(1+\xi)^{s-1}\frac{1}{p}< s \Big(1+\frac{1}{s}\Big)^{s-1}\frac{1}{p}
  <\frac{e s}{p}.
  \end{equation}Since $\log (1+x)\leq x$ for any $x\geq 0$, we obtain
  \begin{align}\label{T1:p.more.s}
  \prod_{p>s}\Big(1+((1+p^{-1})^{s}-1)g(p)\Big)&\leq
  \prod_{p>s}\Big(1+\frac{e s g(p)}{p}\Big)\notag\\
  &\leq
  \textup{exp} \Big(es\sum_{p>s} \frac{g(p)}{p}\Big)\leq
  \textup{exp}(e L s).
  \end{align} If $p\leq s$, then
  \[
  1+((1+p^{-1})^{s}-1)g(p)<1+c_{0}(1+p^{-1})^{s}\leq (1+c_{0})(1+p^{-1})^{s}
  \leq (1+c_{0}) \textup{exp}\Big(\frac{s}{p}\Big).
  \] By Mertens' theorem we have
  \begin{align}
  \prod_{p\leq s}\Big(1+((1+p^{-1})^{s}-1)g(p)\Big)&\leq (1+c_{0})^{s} \textup{exp}\Big(\sum_{p\leq s}\frac{s}{p}\Big)\notag\\
  &\leq  (1+c_{0})^{s} \textup{exp}(s\log \log (s+2) + c_{3}s),\label{T1:p.less.s}
  \end{align}where $c_{3} >0$ is an absolute constant. From \eqref{T1:basic}, \eqref{T1:p.more.s}, and \eqref{T1:p.less.s} we obtain
  \[
  \sum_{n=1}^{N} \Big(\frac{a_n}{\varphi (a_n)}\Big)^{s}\leq \textup{exp}(s\log \log (s+2) + Cs) N,
  \]where $C$ is a positive constant depending only on $\gamma$, $L$, $c_0$, and $\alpha$. The inequality \eqref{T1:EQ.1} is proved.

  Now we prove the inequality \eqref{T1:EQ.2}. We have
  \begin{align*}
  \textup{exp}(s\log \log (s+2) + Cs) N &\geq \sum_{n=1}^{N} \Big(\frac{a_n}{\varphi (a_n)}\Big)^{s}\geq
  \sum_{n\leq N:\, a_{n}/\varphi(a_{n})>t} \Big(\frac{a_n}{\varphi (a_n)}\Big)^{s}\\
  &\geq t^{s} \sum_{n\leq N:\, a_{n}/\varphi(a_{n})>t} 1.
  \end{align*}We obtain
  \begin{equation}\label{T1:RASPRED}
  \#\Big\{n\leq N: \frac{a_{n}}{\varphi (a_{n})}>t\Big\}\leq \textup{exp}(s\log \log (s+2) +Cs -s\log t)N.
  \end{equation}We take
  \[
  s= [\textup{exp}(t e^{- (C+1)})] + 1.
  \] Then
  \[
  s+2 = \textup{exp}(t e^{- (C+1)}) + 3 -\theta,\text{ where }\theta:= \{ \textup{exp}(t e^{- (C+1)})\},
  \] and
  \begin{align*}
  \log (s+2) &= t e^{- (C+1)} + \log \Big(1 + \frac{3-\theta}{\textup{exp}(t e^{- (C+1)})}\Big)\\
  &= t e^{- (C+1)} + R_1,\qquad 0< R_1 \leq \frac{3}{\textup{exp}(t e^{- (C+1)})}.
  \end{align*}We get
  \begin{align*}
  \log \log (s+2) &= \log t - (C+1) + \log \Big(1+ \frac{e^{(C+1)}R_1}{t}\Big)\\
  &= \log t - (C+1) + R_2,\qquad 0< R_2 \leq \frac{3e^{C+1}}{t\,\textup{exp}(t e^{- (C+1)})}.
  \end{align*} Therefore
  \[
  s\log \log (s+2) + Cs - s\log t = -s + s R_2,
  \]where
  \[
  0< s R_2 \leq \frac{3e^{C+1} (\textup{exp}(t e^{- (C+1)}) +1)}{t\,\textup{exp}(t e^{- (C+1)})}
  \leq \frac{6 e^{C+1}}{t}.
  \] We obtain
  \begin{equation}\label{T1:ESTIMATE}
  s\log \log (s+2) + Cs - s\log t \leq -\textup{exp}(t e^{- (C+1)}) +  \frac{6 e^{C+1}}{t} \leq
  - \textup{exp} (t e^{- (C+1)}/2),
  \end{equation}if $t \geq t_{0}$ (here $t_{0}= t_0 (C) = t_{0}(\gamma,L,c_{0}, \alpha)$ is a positive constant depending only on $\gamma$, $L$, $c_0$, and $\alpha$).

   We set
  \[
  c_{2}= e^{-(C+1)}/2,\qquad c_{1}=\max(\textup{exp}(\textup{exp}(c_{2}t_{0})),1).
  \]Then $c_{1}$ and $c_{2}$ are positive constants depending only on $\gamma$, $L$, $c_0$, and $\alpha$. From \eqref{T1:RASPRED} and \eqref{T1:ESTIMATE} we obtain
  \[
  \#\Big\{n\leq N: \frac{a_{n}}{\varphi (a_{n})}>t\Big\}\leq \textup{exp} (-\textup{exp}(c_2 t)) N,
  \] if $t > t_0$. If $0< t\leq t_{0}$, then
  \begin{align*}
  \#\Big\{n\leq N: \frac{a_{n}}{\varphi (a_{n})}>t\Big\}
  &\leq N \leq N\, \frac{\textup{exp}(\textup{exp}(c_{2}t_{0}))}{\textup{exp}(\textup{exp}(c_{2}t))}\\
  &\leq c_{1}\textup{exp}(-\textup{exp}(c_{2}t)) N.
  \end{align*}We obtain
  \[
  \#\Big\{n\leq N: \frac{a_{n}}{\varphi (a_{n})}>t\Big\}\leq
  c_{1}\textup{exp}(-\textup{exp}(c_{2}t)) N
  \]for any $t>0$, and the inequality \eqref{T1:EQ.2} is proved. This completes the proof of Theorem \ref{T1}.
 \end{proof}

 \section{Proof of Theorems \ref{C2} and \ref{C_LIN}}

 \begin{proof}[Proof of Theorem \ref{C2}.] We set $a_{n}=f(n)$. It is clear that $N=[x]$. There exists a positive constant $\tau$ depending only on $f$ such that $f(n)\leq \tau n^{d}$ for any positive integer $n$. We take $M=\tau x^{d}$ and $\alpha= 1/2$. We assume that $x\geq x_{0}$, where $x_{0}>0$ is a large constant depending only on $f$. We have
 \begin{equation}\label{C2:y.EST}
 y=(\log M)^{\alpha}= (d\log x + \log \tau)^{1/2}\leq (\log x)^{3/4},
 \end{equation}if $x_{0}$ is large enough. If $n\in \mathcal{D}$, then
 \begin{equation}\label{C2:n.EST}
 n\leq \prod_{p\leq y}p \leq \textup{exp}(2y)\leq \textup{exp}(2 (\log x)^{3/4})\leq \sqrt{x}< [x] =N.
 \end{equation}We need the following result.
 \begin{lemma}\label{L2}
Let $d$ and $m$ be positive integers. Let
\[
f(x)=\sum_{i=0}^{d} b_{i} x^{i},
\]where $b_{0},\ldots, b_d$ are integers with $(b_{0},\ldots, b_d, m)=1$. Then
\[
\rho (f, m)\leq c d m^{1-1/d},
\]where $c>0$ is an absolute constant.
\end{lemma}
\begin{proof}
This is \cite[Theorem 2]{Konyagin}.
\end{proof}

 By Lemma \ref{L2} we have
 \begin{align*}
 \omega (n)=\#\{k\leq N&: f(k)\equiv 0\text{ \textup{(}mod $n$\textup{)}}\}\leq \rho(f,n) \Big(\frac{N}{n}+1\Big)
 \leq 2\rho(f,n)\,\frac{N}{n}\\
 &\leq (2 c) d n^{1-1/d}\,\frac{N}{n}=\frac{\gamma N}{n^{1/d}},
 \end{align*}where $\gamma=2cd$. The function $g(n)=n^{-1/d}$ is multiplicative. Also, $g(p)=p^{-1/d}\leq  1$ for any prime $p$ (hence, we can take $c_{0}=1$) and
 \[
 L=\sum_{p}\frac{g(p)}{p}=\sum_{p}\frac{1}{p^{1+1/d}}<\infty.
 \]By Theorem \ref{T1} there exist positive constants $C$, $c_{1}$, and $c_{2}$ (depending only on $d$) such that \eqref{C2:EQ.1} and \eqref{C2:EQ.2} hold.

 Suppose that $1 \leq x < x_{0}$. For any $n\leq x$ we have
 \[
 \frac{f(n)}{\varphi(f(n))}\leq f(n)\leq \tau n^d\leq \tau x_{0}^{d}= A,
 \] where $A=A(f)>0$ is a constant depending only on $f$. Hence,
 \[
 \sum_{n\leq x} \Big(\frac{f(n)}{\varphi (f(n))}\Big)^{s}\leq A^{s} x\leq
 \textup{exp}(s\log\log (s+2) +Bs)x,
 \]where $B=\log A$ is a positive constant depending only on $f$.

 We set
 \[
 \lambda(t)=\#\Big\{n\leq x: \frac{f(n)}{\varphi(f(n))}>t\Big\}.
\] If $0< t \leq A$, then
\[
\lambda (t) \leq [x]\leq x_0 \leq x_0 \frac{\textup{exp}(\textup{exp} (c_2 A))}{\textup{exp}(\textup{exp} (c_2 t))}
= \frac{b_1}{\textup{exp}(\textup{exp} (c_2 t))}\leq \frac{b_1}{\textup{exp}(\textup{exp} (c_2 t))} x,
\] where $b_1 = x_0\,\textup{exp}(\textup{exp} (c_2 A))$ is a positive constant depending only on $f$. If $t> A$, then
\[
\lambda (t)=0 \leq \frac{b_1}{\textup{exp}(\textup{exp} (c_2 t))} x.
\]

We put
\[
\widetilde{C} = \max(C, B),\qquad \widetilde{c}_{1} = \max(c_1, b_1),\qquad \widetilde{c}_{2}= c_2.
\]Then $\widetilde{C}$, $\widetilde{c_1}$ and $\widetilde{c}_{2}$ are positive constants depending only on $f$, and for any $x\geq 1$, $s\in \mathbb{N}$, and $t>0$ we have
\[
 \sum_{n\leq x} \Big(\frac{f(n)}{\varphi (f(n))}\Big)^{s}\leq
  \textup{exp}(s\log\log (s+2) +\widetilde{C}s)x
 \]and
 \[
 \#\Big\{n\leq x: \frac{f(n)}{\varphi(f(n))}>t\Big\} \leq \frac{\widetilde{c}_1}{\textup{exp}(\textup{exp} (\widetilde{c}_2 t))} x.
 \]Theorem \ref{C2} is proved.

 \end{proof}

 \begin{proof}[Proof of Theorem \ref{C_LIN}.] We assume that $x\geq x_{0}$, where $x_{0}$ is a large absolute constant. We set
 \[
 \Omega = \{ b\in \mathbb{Z}: |b| \leq \eta \log x, b\neq b_1,\ldots, b_k\}.
 \]Since $\varphi(mn) \geq \varphi(m)\varphi(n)$ for all positive integers $m$ and $n$, and $a^{k+1}/\varphi(a^{k+1})=a/\varphi(a)$, we obtain
 \begin{equation}\label{C2.GENERAL}
 S= \sum_{b\in \Omega} \bigg(\frac{a^{k+1}|f(b)|}{\varphi(a^{k+1}|f(b)|)}\bigg)^{s}
 \leq \Big(\frac{a}{\varphi(a)}\Big)^{s} \sum_{b\in \Omega} \bigg(\frac{|f(b)|}{\varphi(|f(b)|)}\bigg)^{s}.
 \end{equation}

  Let $b\in \Omega$. Hence, $|b-b_i| \leq 2 \log x$ for all $1\leq i \leq k$, and therefore $|f(b)| \leq (2\log x)^{k}$. Suppose that $k\geq \log\log x$.  Since $n/\varphi (n) \leq c \log\log (n+2)$ for any positive integer $n$, where $c$ is an absolute positive constant, we obtain
  \[
  \frac{|f(b)|}{\varphi (|f(b)|)} \leq c (\log k + \log\log (4\log x))\leq c_1 \log k.
  \] Hence,
  \[
  S \leq \Big(\frac{a}{\varphi(a)}\,c_1 \log k\Big)^{s}\# \Omega \leq
  \Big(c_2\,\frac{a}{\varphi(a)} \log k\Big)^{s} \eta \log x,
  \]where $c_2$ is an absolute positive constant.

  Suppose that $2 \leq k < \log\log x$. We take $\alpha =1/4$, $M = (2\log x)^{k}$. Then
  \[
  y= (\log M)^{\alpha} \leq 2 (\log\log x)^{1/2}.
  \] If $n\in \mathcal{D}$, then
 \[
 n\leq \prod_{p\leq y}p \leq \textup{exp}(2y)\leq \textup{exp}(4 (\log\log x)^{1/2})\leq \eta \log x.
 \]For $n\in \mathcal{D}$, we have
 \begin{align*}
 \omega (n) = \#\{b\in \Omega: f(b) \equiv 0 \text{ (mod $n$)}\}&\leq \Big(\frac{2 \eta \log x}{n} + 1\Big)\prod_{p|n} \min (p,k)\\
 &\leq \frac{3 \eta \log x}{n}\prod_{p|n} \min (p,k).
 \end{align*} We set $g(1)=1,$
 \[
 g(p^{\beta}) = \begin{cases}
 \min(p,k)/p, &\text{if $\beta=1$;}\\
 1,                   &\text{if $\beta \geq 2$};
                \end{cases}
 \] and $g(p_{1}^{\beta_1} \dots p_{r}^{\beta_{r}}) = g(p_{1}^{\beta_1})\dots g(p_{r}^{\beta_r})$, where $p_1, \ldots, p_r$ are pairwise distinct prime numbers. Then $g(n)$ is a multiplicative function and
 \[
 \omega(n)\leq (3\eta \log x) g(n)
 \] for any $n\in \mathcal{D}$. Since $g(p)>0$ for any prime $p$, from Theorem \ref{T_MULTI} we obtain
 \begin{align}
 \sum_{b\in \Omega} \bigg(\frac{|f(b)|}{\varphi(|f(b)|)}\bigg)^{s}\leq c^{s}(\eta \log x)
 \prod_{p\leq y}
\Big(1 + ((1+p^{-1})^{s}-1)g(p)\Big)\notag\\
\leq c^{s}(\eta \log x)
 \prod_{p}
\Big(1 + ((1+p^{-1})^{s}-1)g(p)\Big)\label{C2.BASIC.SUM}
 \end{align} (here $c>0$ is an absolute constant). By \eqref{C2.GENERAL} and \eqref{C2.BASIC.SUM} we have
 \[
 S \leq \Big(c\,\frac{a}{\varphi (a)}\Big)^{s} \eta \log x \prod_{p}
\Big(1 + ((1+p^{-1})^{s}-1)g(p)\Big).
 \]

 Suppose that $s \leq k$ (we recall that $2 \leq k < \log\log x$). If $p> k$, then $p>s$ and \eqref{T1.LAGRANGE} holds. Since $g(p) = k/p$, we obtain
 \[
 \prod_{p > k}
\Big(1 + ((1+p^{-1})^{s}-1)g(p)\Big)\leq \prod_{p>k}
\Big(1 + \frac{e k s}{p^{2}}\Big) \leq c_{1}^{s},
 \] where $c_1 >0$ is an absolute constant. If $p\leq k$, then $g(p) =1$ and we have
 \[
 1 + ((1+p^{-1})^{s}-1)g(p) = (1+p^{-1})^{s}\leq \textup{exp}\Big(\frac{s}{p}\Big).
 \] Hence
 \[
 \prod_{p \leq k}
\Big(1 + ((1+p^{-1})^{s}-1)g(p)\Big)\leq \textup{exp}\Big(\sum_{p \leq k} \frac{s}{p}\Big)\leq (c_2 \log k)^{s},
 \]where $c_2>0$ is an absolute constant. The inequality \eqref{C_LIN:EQ.1} is proved.

 Suppose that $s> k$ and $2 \leq k < \log\log x$. Since $g(p) \leq k/p < s/p$ for any prime $p$, we obtain
 \[
 \prod_{p > s}
\Big(1 + ((1+p^{-1})^{s}-1)g(p)\Big)\leq \prod_{p>s}
\Big(1 + \frac{e s^{2}}{p^{2}}\Big) \leq c_{3}^{s}.
 \] Since $g(p) \leq 1$ for any prime $p$, we have
\[
 \prod_{p \leq s}
\Big(1 + ((1+p^{-1})^{s}-1)g(p)\Big)\leq \textup{exp}\Big(\sum_{p \leq s} \frac{s}{p}\Big)\leq (c_4 \log s)^{s}.
 \]The inequality \eqref{C_LIN:EQ.2} is proved.

 Suppose that $3\leq x < x_0$. For any $b\in \Omega$ we have $|f(b)|\leq (2\log x)^{k} \leq A_{0}^{k}$, where $A_0 = 2 \log x_0$, and
 \[
 \frac{|f(b)|}{\varphi(|f(b)|)}\leq c \log\log (|f(b)|+2)\leq A_{1} \log k,
 \]where $A_{1}>0$ is an absolute constant. Since
 \[
 \#\Omega\leq 2\eta \log x + 1 \leq 2\log x_{0}+1
 \]and
 \[
 \eta \log x \geq (\log x)^{1/10}\geq (\log 3)^{1/10}>1,
 \] by \eqref{C2.GENERAL} we have
 \[
 S \leq \Big(A_{1} \frac{a}{\varphi (a)}\log k\Big)^{s}\#\Omega
 \leq \Big(A_{2} \frac{a}{\varphi (a)}\log k\Big)^{s}\eta \log x,
 \]where $A_2 >0$ is an absolute constant. This completes the proof of Theorem \ref{C_LIN}.

 \end{proof}

 \section{Proof of Theorem \ref{T.f(p)}}

 \begin{proof} We set $a_{p}=f(p)$. It is clear that $N=\pi (x)$. There exists a positive constant $\tau$ depending only on $f$ such that $f(p)\leq \tau p^{d}$ for any prime $p$. We take $M=\tau x^{d}$ and $\alpha= 1/2$. We assume that $x\geq x_{0}$, where $x_{0}>0$ is a large constant depending only on $f$. By \eqref{C2:y.EST} and \eqref{C2:n.EST} we have $y= (\log M)^{\alpha}\leq (\log x)^{3/4}$ and $n \leq \sqrt{x}$ for any $n\in \mathcal{D}$.

 We need the following result.

 \begin{lemma}[The Brun - Titchmarsh inequality]\label{L:Brun.T}
 If $1\leq k < x$ and $(a,k) = 1$, then
 \[
 \pi (x; k, a) < \frac{ 3 x}{\varphi (k) \log (x/k)}.
 \]
  \end{lemma}
  \begin{proof}
  This is \cite[Theorem 3.8]{Halberstam.Richert}.
  \end{proof}

 Fix $n\in \mathcal{D}$, and let $\alpha_{1}, \ldots, \alpha_{l}$ be all solutions of the congruence $f(x) \equiv 0$ (mod $n$). We have
 \[
 l=\rho (f, n) = \prod_{p|n} \rho (f,p)= \prod_{\substack{p|n\\ p\nmid b_{d}}} \rho (f,p)\cdot
 \prod_{\substack{p|n\\ p| b_{d}}} \rho (f,p) = A\cdot B.
 \]

 For any prime $p\nmid b_{d}$, by Lagrange's theorem we have $\rho (f, p) \leq \min (p,d)$. Trivially $\rho (f, p) \leq p$ for any prime $p$. Hence,
 \[
 A \leq \prod_{\substack{p|n\\ p\nmid b_{d}}} \min (p, d) \leq \prod_{p|n} \min (p, d),\qquad
 B \leq \prod_{\substack{p|n\\ p| b_{d}}} p \leq \prod_{p|b_{d}} p = c(f).
 \] We obtain
 \begin{equation}\label{T.f(p):l.EST}
 l\leq c(f) \prod_{p|n} \min (p, d),
 \end{equation}where $c(f)>0$ is a constant depending only on $f$.

 We have
 \begin{equation}\label{T.f(p):omega.EST}
 \omega (n) = \#\{p\leq x: f(p)\equiv 0\text{ (mod $n$)}\}= \sum_{i=1}^{l} \pi (x; n, \alpha_{i}).
 \end{equation}Fix $i\in \{1,\ldots, l\}$ and suppose that $(\alpha_i, n)=1$. We have
 \[
 \log \Big(\frac{x}{n}\Big)\geq \log(\sqrt{x})= \frac{1}{2}\log x.
 \]Applying Lemma \ref{L:Brun.T}, we obtain
 \begin{equation}\label{T.f(p):prime.progr.EST}
 \pi (x; n, \alpha_i)\leq \frac{6x}{\varphi (n) \log x}.
 \end{equation}

 Suppose that $(\alpha_i, n)> 1$. Hence,
 \[
 \pi (x; n, \alpha_i)\leq 1 \leq \frac{6x}{\varphi (n) \log x},
 \]since
 \[
 \frac{6x}{\varphi (n) \log x} \geq \frac{6x}{n \log x}\geq \frac{6\sqrt{x}}{\log x}> 1,
 \]if $x_0$ is large enough. We see that the inequality \eqref{T.f(p):prime.progr.EST} holds in both cases $(\alpha_i, n)=1$ and $(\alpha_i, n)>1$. From \eqref{T.f(p):l.EST} -- \eqref{T.f(p):prime.progr.EST} we get
 \[
 \omega (n) \leq \frac{6x l}{\varphi (n) \log x} \leq \frac{\gamma \pi(x)}{\varphi(n)} \prod_{p|n} \min (p,d),
 \]where $\gamma= \gamma (f)>0$ is a constant depending only on $f$.

 We set $g(1)=1,$
 \[
 g(p^{\beta}) = \begin{cases}
 \min(p,d)/(p-1), &\text{if $\beta=1$;}\\
 1,                   &\text{if $\beta \geq 2$};
                \end{cases}
 \] and $g(p_{1}^{\beta_1} \dots p_{r}^{\beta_{r}}) = g(p_{1}^{\beta_1})\dots g(p_{r}^{\beta_r})$, where $p_1, \ldots, p_r$ are pairwise distinct prime numbers. Then $g(n)$ is a multiplicative function and
 \[
 \omega(n)\leq \gamma \pi(x)g(n)
 \] for any $n\in \mathcal{D}$.

 Since
 \[
 g(p) = \frac{\min (p,d)}{p-1}\leq \frac{p}{p-1}\leq 2
 \] for any prime $p$ and
 \[
 L=\sum_{p}\frac{g(p)}{p}= \sum_{p\leq d}\frac{1}{p-1} + \sum_{p>d}\frac{d}{p(p-1)}<\infty,
 \]by Theorem \ref{T1} there exist positive constants $C$, $c_1$ and $c_2$ depending only on $f$ such that the inequalities \eqref{T.f(p):EQ.1} and \eqref{T.f(p):EQ.2} hold.

 By arguing as in the end of the proof of Theorem \ref{C2}, we find that there exist positive constants $\widetilde{C}$, $\widetilde{c}_1$, and $\widetilde{c}_2$ depending only on $f$ such that for any $x\geq 2$, $s\in \mathbb{N}$, and $t>0$ we have
\[
 \sum_{p\leq x} \Big(\frac{f(p)}{\varphi (f(p))}\Big)^{s}\leq
  \textup{exp}(s\log\log (s+2) +\widetilde{C}s)\pi(x)
 \]and
 \[
 \#\Big\{p\leq x: \frac{f(p)}{\varphi(f(p))}>t\Big\} \leq \frac{\widetilde{c}_1}{\textup{exp}(\textup{exp} (\widetilde{c}_2 t))} \pi(x).
 \]Theorem \ref{T.f(p)} is proved.

 \end{proof}

\section{Acknowledgements}
The author is grateful to Mikhail R. Gabdullin for useful discussions and to the anonymous referee for useful comments.

This research was supported by Russian Science Foundation, grant 20-11-20203, https://rscf.ru/en/project/20-11-20203/.


\begin{thebibliography}{99}

  \bibitem{Halberstam.Richert}
  H.~Halberstam and H.\,E.~Richert. \emph{Sieve methods.} L.M.S. monographs. Academic Press, 1974.

  \bibitem{Konyagin}
 S.\,V.~Konyagin. On the number of solutions of an $n$th degree congruence with one unknown. \emph{Sb. Math.}, \textbf{37} (1980), no. 2, 151--166.

 \bibitem{Maynard}
  J.~Maynard. Dense clusters of primes in subsets. \emph{Compos. Math.}, \textbf{152} (2016), no. 7, 1517--1554.

  \bibitem{Radomskii.Izv}
  A.\,O.~Radomskii. On Romanoff's theorem. \emph{Izv. Math.}, \textbf{87} (2023), no. 1, 113-153.

















  \end{thebibliography}
\end{document}